\documentclass[11pt]{amsart}

\usepackage[margin=1.15in]{geometry}
\usepackage{amsmath,amssymb,amsthm,mathrsfs,mathtools}
\usepackage{enumitem}
\usepackage{hyperref}
\usepackage{microtype}
\usepackage{xcolor}

\newtheorem{theorem}{Theorem}[section]
\newtheorem{proposition}[theorem]{Proposition}
\newtheorem{lemma}[theorem]{Lemma}
\newtheorem{corollary}[theorem]{Corollary}
\theoremstyle{definition}
\newtheorem{definition}[theorem]{Definition}
\newtheorem{remark}[theorem]{Remark}

\newcommand{\N}{\mathbb{N}}
\newcommand{\A}{\mathcal{A}}
\newcommand{\full}{\Sigma_{\A}}
\newcommand{\fin}{\Sigma_{\A}^{\mathrm{fin}}}
\newcommand{\emptyword}{\tilde{0}} 

\title[Lipschitz Shadowing over infinite alphabets]{Shadowing and Lipschitz Shadowing in Symbolic Dynamics: Finite vs. Infinite Alphabets}

\author[D.~Gon\c{c}alves]{Daniel Gon\c{c}alves}
\address{Departamento de Matem\'atica, Universidade Federal de Santa Catarina,
88040-970 Florian\'opolis SC, Brazil}
\email{daemig@gmail.com}

\author[S. Meneghel]{Sofia Meneghel Silva}
\address{Departamento de Matem\'atica, Universidade Federal de Santa Catarina,
88040-970 Florian\'opolis SC, Brazil}
\email{fifia.meneghel@gmail.com}



\thanks{The first author was partially supported by Capes-Print Brazil, Conselho Nacional de Desenvolvimento Cient\'ifico e Tecnol\'ogico (CNPq) - Brazil, and Funda\c{c}\~ao de Amparo \`a Pesquisa e Inova\c{c}\~ao do Estado de Santa Catarina (FAPESC). The second author was supported by Funda\c{c}\~ao de Amparo \`a Pesquisa e Inova\c{c}\~ao do Estado de Santa Catarina (FAPESC) }

\date{}

\begin{document}

\begin{abstract}

We point out a basic dichotomy between the shadowing and Lipschitz shadowing properties for one-sided
shift spaces in two infinite-alphabet frameworks: the classical product-topology model
$X\subseteq A^{\N}$ and the compact Ott--Tomforde--Willis (OTW) model obtained by adjoining finite words.
In the product-topology setting, for the natural class of prefix ultrametrics, shadowing and Lipschitz shadowing
coincide.  However, since $A^{\N}$ is non-compact when $A$ is countably infinite,
it remains unclear whether Lipschitz shadowing is stable under arbitrary uniformly equivalent
changes of compatible metric in the product-topology model.
In contrast, for OTW shift spaces the topology admits a canonical family of compatible ultrametrics
indexed by enumerations of finite words, and these metrics are all uniformly equivalent.  Using the
Deaconu--Renault viewpoint and known shadowing results for local homeomorphisms on zero-dimensional
compact spaces, we show that the OTW full shift has the shadowing property for every OTW metric.
Nevertheless, Lipschitz shadowing can depend on the chosen OTW metric even within this fixed uniform
equivalence class: we construct two uniformly equivalent OTW ultrametrics on the full shift for which
Lipschitz shadowing holds in one case and fails in the other.  Thus the OTW compactification provides a
compact infinite-alphabet setting where the metric dependence of Lipschitz shadowing can be resolved
explicitly, in sharp contrast with what is currently known for the product-topology model.

\end{abstract}

\maketitle

\section{Introduction}

Symbolic dynamics over a finite alphabet is a central topic in dynamical systems, with deep
connections to smooth dynamics, coding, and operator algebras; see, for instance,
\cite{LindMarcus1995,Kitchens1998}.
Over a countably infinite alphabet, however, several inequivalent frameworks coexist.
A classical approach studies countable Markov shifts (typically non-compact) and their
thermodynamic formalism, where fine recurrence properties play a decisive role
\cite{Sarig1999,BissacotGaribaldi2010,BissacotFreire2014}.
From another viewpoint, one may consider general one-sided subshifts $X\subseteq A^{\N}$ with the
product topology (and the standard prefix ultrametric). In this setting compactness is lost and
quantitative metric notions can become sensitive to the choice of compatible metric.

In contrast, Ott--Tomforde--Willis (OTW) introduced a compact, totally disconnected model of one-sided
shift spaces over arbitrary discrete alphabets by adjoining finite words and working with a basis of
generalized cylinders \cite{OTW2014}. This OTW compactification has proved particularly well-suited to
connections with groupoids and $C^*$-algebras, and it has stimulated a growing literature on the topology
and dynamics of infinite-alphabet shifts. Moreover, recent results show that isometric conjugacy of
countable Markov shifts can be realized as conjugacy of the associated OTW subshifts,
see \cite{BoavaCastroGoncalvesvanWyk2025}.
Several dynamical and structural phenomena familiar from finite alphabets require substantial care in the
OTW setting. For instance, higher-step presentations need not collapse as in the finite-alphabet case:
already for one-sided OTW shifts there are $(M\!+\!1)$-step shift spaces that are not conjugate to
any $M$-step shift space \cite{GoncalvesRoyer2015}.
Morphisms also behave more subtly: sliding block codes between OTW shift spaces and variants of the
Curtis--Hedlund--Lyndon theorem have been investigated in
\cite{GoncalvesSobottkaStarling2016,GoncalvesSobottka2017},
and further aspects of reversibility and image sets for shift morphisms over general discrete alphabets
appear in \cite{CamposRomeroVivas2022}.
From a different viewpoint, certain classes of OTW shift spaces arise naturally as edge shift spaces of
ultragraphs, linking symbolic dynamics with graph and ultragraph $C^*$-algebras
\cite{GoncalvesRoyer2017,GoncalvesRoyer2019IMRN},
and this perspective has motivated dynamical results such as Li--Yorke chaos and the existence of
scrambled sets in compact OTW/ultragraph models
\cite{GoncalvesUggioni2019Chaos,GoncalvesUggioni2020,RainesUnderwood2016}.
More recently, entropy theory for local homeomorphisms has been developed with applications to
infinite-alphabet shifts \cite{GoncalvesRoyerTasca2024}, and new operator-algebraic invariants for
one-sided subshifts over arbitrary alphabets were obtained in
\cite{BoavaCastroGoncalvesvanWyk2024,BoavaCastroGoncalvesvanWyk2025}.
Related compact models beyond OTW also exist, for instance blur shift spaces \cite{Almeida2021Blur}.

A recurring theme in this area is to identify which finite-alphabet dynamical properties remain
meaningful and stable in infinite-alphabet models. In this paper we focus on shadowing, a
classical notion capturing the reliability of approximate orbits (pseudo-orbits), and on its quantitative
variant Lipschitz shadowing.
In the product-topology model $X\subseteq A^{\N}$ equipped with the standard prefix ultrametric
$d_{\mathrm{prod}}(x,y)=2^{-N(x,y)}$, we show that shadowing and Lipschitz shadowing coincide (indeed, the
same holds for the natural class of prefix ultrametrics).  This uses a characterization of
shadowing by finite order for one-sided countable-alphabet subshifts, see \cite{{MR4256132}}, and an explicit Lipschitz shadowing
estimate.  In the finite-alphabet case, one can also compare with Sakai's result that for positively
expansive maps on compact spaces, shadowing is equivalent to the existence of a compatible metric
with Lipschitz shadowing \cite[Theorem~1]{Sakai2003}.  However, when $A$ is countably infinite the space
$A^{\N}$ is non-compact, and compatible metrics inducing the product topology need not be uniformly
equivalent.  This leads naturally to the following unresolved issue: if $d$ is uniformly equivalent to
$d_{\mathrm{prod}}$ and $(X,\sigma,d)$ has shadowing, must $(X,\sigma,d)$ have Lipschitz shadowing?  We do
not know the answer.

Our main point is that the OTW compactification provides a compact infinite-alphabet setting where the
corresponding ``uniform invariance'' question can be settled.  OTW associate to each enumeration of the
finite words a compatible ultrametric; different enumerations yield uniformly equivalent metrics, so the
(topological) shadowing property is robust under changing the OTW metric.  Using the Deaconu--Renault
(local homeomorphism) viewpoint and existing shadowing results for Deaconu--Renault systems, we show that
the OTW full shift has the shadowing property for every OTW metric.  In contrast, Lipschitz
shadowing is genuinely metric-dependent in the OTW world: we construct two uniformly equivalent OTW
ultrametrics on the full shift such that Lipschitz shadowing holds in one case and fails in the other.
Thus, unlike the product-topology and finite-alphabet settings, the OTW model exhibits a sharp separation
between shadowing and Lipschitz shadowing already at the level of the full shift.

The paper is organized as follows.
In Section~\ref{shadowingprelim} we recall pseudo-orbits, shadowing and Lipschitz shadowing, and prove
the two basic invariance principles used throughout: shadowing is invariant under uniform equivalence
(Proposition~\ref{prop:shadowing-uniform}), while Lipschitz shadowing is invariant under bi-Lipschitz
changes of metric (Proposition~\ref{prop:lipschitz-bilip}).
In Section~\ref{finite alphabet sec} we study the product-topology model: we prove that for the standard
product metric (and for prefix ultrametrics) shadowing and Lipschitz shadowing coincide, and we record
the open question of stability under uniformly equivalent changes of metric in the non-compact setting.
Next, in Section~\ref{OTWfullsec} we introduce the OTW full shift and the OTW metrics coming from
enumerations of finite words (Definition~\ref{def:otwmetric}), prove that these metrics are ultrametrics
(Proposition~\ref{prop:otw-ultrametric}) inducing the OTW topology, and record that OTW metrics associated
to different enumerations are uniformly equivalent (Lemma~\ref{lem:uniform}).
In Section~\ref{OTWshadowing} we establish shadowing for the OTW full shift
(Proposition~\ref{prop:shadowingfull}) and, by uniform invariance, for every OTW metric.
Finally, in Section~\ref{OTW lip shad} we analyze Lipschitz shadowing in the OTW setting: we construct a
``bad'' OTW metric (via Lemma~\ref{lem:bad-enum}) for which Lipschitz shadowing fails
(Theorem~\ref{thm:lipschitz-fails}), and also a ``good'' OTW metric arising from a
prefix--shift-compatible enumeration (Definition~\ref{def:pscomp}) for which Lipschitz shadowing holds
(Proposition~\ref{prop:lipschitz-holds}). These two OTW metrics are uniformly equivalent, isolating the
mechanism behind the non-invariance of Lipschitz shadowing under uniform equivalence in the OTW setting.

\section{Shadowing, Lipschitz shadowing, and uniform equivalence}\label{shadowingprelim}

Deaconu--Renault systems were introduced independently by Deaconu \cite{Deaconu1995} and Renault \cite{Renault2000}. In this section we define Deaconu--Renault systems and the notions of shadowing and Lipschitz shadowing for such systems. We then show that shadowing is invariant under uniformly equivalent metrics, while Lipschitz shadowing is invariant under bi-Lipschitz equivalent metrics.

\begin{definition}[Deaconu--Renault system] 
    A \textit{Deaconu--Renault system} is a pair $(X,\sigma)$ consisting of a locally compact Hausdorff space $X$, and a local homeomorphism $\sigma: Dom(\sigma)\longrightarrow Im(\sigma)$ from an open set $Dom (\sigma)\subseteq X$ to an open set $Im(\sigma)\subseteq X$. 
    
    Inductively define $Dom(\sigma^n):=\sigma^{-1}(Dom(\sigma^{n-1}))$, so each $\sigma^n:Dom(\sigma^n)\longrightarrow Im(\sigma^n)$ is a local homeomorphism and $\sigma^m\circ \sigma^n=\sigma^{m+n}$ on $Dom(\sigma^{m+n})$.
\end{definition}

\begin{definition}[Pseudo-orbits and shadowing]
Let $(X,\sigma)$ be a Deaconu--Renault system, with $X$ a metric space.
\begin{enumerate}[label=(\alph*)]
\item For $\delta>0$, a sequence $(x^n)_{n\ge 0}$ is a \emph{$\delta$-pseudo-orbit} in $(X,\sigma)$ if
\[ d\bigl(\sigma(x^n),x^{n+1}\bigr) < \delta \quad \text{for all } n\ge 0,\]
where it is implicit that $x^n\in Dom(\sigma)$ when $\sigma(x^n)$ is written.
\item Given $\varepsilon>0$, we say that $x\in X$ \emph{$\varepsilon$-shadows} a sequence $(x^n)_{n \ge 0}$ in $X$ if
\[ d\bigl(\sigma^i(x),x^i\bigr) < \varepsilon \quad \text{for all } i\ge 0.\]
Notice that this implies that $x\in Dom(\sigma^i)$ for all $i$.
\item We say the Deaconu--Renault system $(X,\sigma)$ has the \textit{shadowing property} if for every $\varepsilon>0$, there exists a $\delta>0$ such that every  $\delta-$pseudo-orbit is $\varepsilon-$shadowed by a point $x\in X$.
\end{enumerate}
\end{definition}

\begin{definition}[Lipschitz shadowing]
Let $(X,\sigma)$ be a Deaconu--Renault system, with $X$ a metric space.
We say that $(X,\sigma)$ has the \emph{Lipschitz shadowing property} if there exist constants $L>0$ and $\delta_0>0$ such that for every $\delta\in(0,\delta_0]$, every $\delta$-pseudo-orbit is $L\delta$-shadowed by some point of $X$.
\end{definition}

\begin{definition}[Uniform equivalence]
Let $d_1$ and $d_2$ be two metrics on $X$.
We say that $d_1$ and $d_2$ are \emph{uniformly equivalent} if the identity maps
\[
\mathrm{id_1}:(X,d_1)\to (X,d_2)
\quad\text{and}\quad
\mathrm{id_2}:(X,d_2)\to (X,d_1)
\]
are uniformly continuous.
\end{definition}

\begin{proposition}[Shadowing is invariant under uniform equivalence]
\label{prop:shadowing-uniform}
Let $(X,\sigma)$ be a Deaconu--Renault system, with $X$ a metric space. Also let $d_1,d_2$ be uniformly equivalent metrics on $X$.
Then $(X,\sigma)$, where $X$ has the metric $d_1$, has shadowing if and only if $(X,\sigma)$, where $X$ has the metric $d_2$, has shadowing.
\end{proposition}

\begin{proof}
This is proved in \cite[Proposition~2.12]{GoncalvesUggioni2025Shadowing}.
\end{proof}

The notion of Lipschitz shadowing is preserved by a Bi-Lipschitz equivalence. This must be well known but we include a proof for completeness. 

 \begin{definition}[Bi-Lipschitz equivalence]
Let $(X,d_1)$ and $(X,d_2)$ be metric spaces on the same underlying set $X$.
We say that $d_1$ and $d_2$ are \textit{bi-Lipschitz equivalent} if there exist constants
$\alpha,\beta>0$ such that
\[
\alpha\, d_2(x,y)\ \le\ d_1(x,y)\ \le\ \beta\, d_2(x,y)
\qquad\text{for all }x,y\in X.
\]
\end{definition}

\begin{proposition}[Lipschitz shadowing is preserved under bi-Lipschitz changes]
\label{prop:lipschitz-bilip}
Let $(X,\sigma)$ be a Deaconu--Renault system, with $X$ a metric space, and let $d_1,d_2$ be bi-Lipschitz equivalent metrics on $X$.
Then $(X,\sigma)$ has the Lipschitz shadowing property with respect to the metric $d_1$ if and only if $(X,\sigma)$ has the Lipschitz shadowing property with respect to the metric $d_2$.
\end{proposition}

\begin{proof}
Assume first that $(X,\sigma)$ has the Lipschitz shadowing property with respect to the metric $d_1$. Then there exist constants
$L>0$ and $\delta_0>0$ such that for every $\delta\in(0,\delta_0]$, every $\delta$-pseudo-orbit
$(x^n)_{n\ge0}$ with respect to $d_1$ is $L\delta$-shadowed (with respect to $d_1$) by some point $x\in X$, i.e.
\[
d_1\bigl(\sigma^n(x),x^n\bigr)\ \le\ L\delta \qquad\text{for all }n\ge0.
\]
Let $\alpha,\beta>0$ satisfy $\alpha d_2\le d_1\le \beta d_2$.

Fix $\delta\in\bigl(0,\delta_0/\beta\bigr]$ and let $(x^n)_{n\ge0}$ be a $\delta$-pseudo-orbit with respect to $d_2$.
Then for all $n\ge0$,
\[
d_1\bigl(\sigma(x^n),x^{n+1}\bigr)\ \le\ \beta\, d_2\bigl(\sigma(x^n),x^{n+1}\bigr)\ <\ \beta\delta\ \le\ \delta_0,
\]
so $(x^n)_{n \ge 0}$ is a $(\beta\delta)$-pseudo-orbit with respect to $d_1$.
Hence there exists $x\in X$ such that
\[
d_1\bigl(\sigma^n(x),x^n\bigr)\ \le\ L(\beta\delta)\qquad\text{for all }n\ge0.
\]
Using $\alpha d_2\le d_1$ we obtain
\[
d_2\bigl(\sigma^n(x),x^n\bigr)\ \le\ \frac{1}{\alpha}\, d_1\bigl(\sigma^n(x),x^n\bigr)
\ \le\ \frac{L\beta}{\alpha}\,\delta
\qquad\text{for all }n\ge0.
\]
Thus $(X,\sigma)$ has Lipschitz shadowing with respect to $d_2$ and with constants $L'=\frac{L\beta}{\alpha}$ and $\delta_0'=\delta_0/\beta$.

The converse implication (from with respect to $d_2$ to with respect to $d_1$) is symmetric.
\end{proof}

\section{Standard product topology: shadowing and Lipschitz shadowing coincide}\label{finite alphabet sec}

In this section we discuss the one-sided shift $(X,\sigma)$ with $X\subseteq A^{\N}$ endowed with the
product topology.  We introduce a large natural class of compatible metrics (prefix ultrametrics)
that satisfy an explicit uniform condition that implies Lipschitz shadowing for shifts of finite order.
We also record an open question on the extent to which Lipschitz shadowing is stable under uniformly
equivalent changes of metric in the non-compact, countable-alphabet setting.

\subsection{The standard product ultrametric and finite order}

Let $A$ be a discrete countable alphabet and let $X\subseteq A^{\N}$ be a one-sided shift space with the
left shift $\sigma\colon X\to X$.  We consider the standard product ultrametric
\begin{equation}\label{eq:prod-metric}
d_{\mathrm{prod}}(x,y)=2^{-N(x,y)},\qquad
N(x,y)=\min\{n\ge 0:x_n\neq y_n\},
\end{equation}
with the convention $2^{-\infty}=0$.

We begin by recording that for one-sided subshifts over a countable alphabet,
shadowing is equivalent to a finite-memory condition (finite order), which extends the classical
finite-alphabet picture.

\begin{theorem}\label{thm:darji-finite-order}
A one-sided shift space $X\subseteq A^{\N}$ has the shadowing property (with respect to $d_{\mathrm{prod}}$)
if and only if $X$ is a shift of finite order, i.e.\ there exists $p\ge 1$ such that membership in $X$
is determined by allowed blocks of length $p$.
\end{theorem}

\begin{proof}
See \cite[Theorem~3.7]{MR4256132}.
\end{proof}

\subsection{A uniform criterion implying Lipschitz shadowing}

We now formulate a convenient sufficient condition for Lipschitz shadowing in the product topology.
It is tailored to shift spaces of finite order and to metrics whose small balls contain a
uniformly long cylinder around each point.

For $n\ge 0$ set
\[
U_n:=\{(x,y)\in X\times X:\ x_0x_1\cdots x_{n-1}=y_0y_1\cdots y_{n-1}\},
\quad (U_0=X\times X).
\]
Given a compatible metric $d$ on $X$, define
\[
\alpha_n:=\sup\{d(x,y):(x,y)\in U_n\},
\qquad
\eta_n:=\sup\{\eta>0:\ d(x,y)<\eta \Rightarrow (x,y)\in U_n\}.
\]

\begin{proposition}\label{prop:finite-order-unif-Lip-sufficient}
Assume that $X$ is a shift of finite order $p$. Let $d$ be a metric on $X$ inducing the product topology.
If there exists $C\ge 1$ such that
\begin{equation}\label{eq:alpha-eta-linear}
\alpha_n \le C\,\eta_{n+p}\qquad\text{for all }n\ge0,
\end{equation}
then $(X,\sigma,d)$ has the Lipschitz shadowing property with Lipschitz constant $L=C$.
\end{proposition}

\begin{proof}
Fix $0<\delta\le \eta_p$ and let $\{x^m\}_{m\ge0}\subseteq X$ be a $\delta$-pseudo-orbit:
$d(\sigma(x^m),x^{m+1})<\delta$ for all $m\ge0$.
Choose $k\ge0$ such that $\eta_{k+p+1}\le \delta < \eta_{k+p}$.
Since $\delta<\eta_{k+p}$, for every $m\ge0$ we have $(\sigma(x^m),x^{m+1})\in U_{k+p}$, i.e.
\[
x^m_{t+1}=x^{m+1}_{t}\qquad\text{for }t=0,1,\dots,k+p-1\ \text{and all }m\ge0.
\]
Define $z\in A^{\N}$ by $z_m:=x^m_0$. Iterating the overlap relations yields, for each $m\ge0$ and each
$t=0,1,\dots,k$, the identity $(\sigma^m(z))_t=x^m_t$, hence $(\sigma^m(z),x^m)\in U_k$ for all $m$.

Since $k+p\ge p$, the same overlap shows that every length-$p$ block of $z$ occurs as a length-$p$ block
of some $x^m\in X$. As $X$ is of order $p$, this implies $z\in X$.

Finally, for each $m\ge0$,
\[
d(\sigma^m(z),x^m)\le \alpha_k \le C\,\eta_{k+p} \le C\,\delta,
\]
using \eqref{eq:alpha-eta-linear} and $\delta<\eta_{k+p}$. Thus $z$ $C\delta$-shadows the pseudo-orbit,
so $(X,\sigma,d)$ has Lipschitz shadowing with $L=C$.
\end{proof}

\subsection{Prefix ultrametrics}

We show below that Condition \eqref{eq:alpha-eta-linear} holds for a large and natural class of compatible metrics.

\begin{definition}[Prefix ultrametric]\label{def:prefix-ultrametric}
A metric $d$ on $X\subseteq A^{\N}$ is a \emph{prefix ultrametric} if there exists a strictly decreasing
sequence $\mathbf r=(r_n)_{n\ge0}\subset(0,\infty)$ with $r_n\downarrow 0$ such that for $x\neq y$,
\[
d(x,y)=r_{N(x,y)},\qquad N(x,y):=\min\{n\ge0:\ x_n\neq y_n\},
\]
and $d(x,x)=0$.
\end{definition}

\begin{lemma}\label{lem:prefix-metric-satisfies-alpha-eta}
If $d$ is a prefix ultrametric, then \eqref{eq:alpha-eta-linear} holds with $C=1$ (for every $p\ge1$).
\end{lemma}

\begin{proof}
Let $d$ be determined by $(r_n)$. If $(x,y)\in U_n$ and $x\neq y$, then $N(x,y)\ge n$, hence
$d(x,y)=r_{N(x,y)}\le r_n$, so $\alpha_n\le r_n$.
Moreover, for $n\ge1$, the condition $d(x,y)<r_{n-1}$ forces $N(x,y)\ge n$, hence $(x,y)\in U_n$, so
$\eta_n\ge r_{n-1}$. Therefore, for any $p\ge1$,
\[
\alpha_n \le r_n \le r_{n+p-1} \le \eta_{n+p},
\]
since $(r_n)$ is decreasing. This is \eqref{eq:alpha-eta-linear} with $C=1$.
\end{proof}

\begin{corollary}\label{cor:prefix-equivalence}
Let $X\subseteq A^{\N}$ be a one-sided shift space (with $A$ countable) and let $d$ be any prefix
ultrametric on $X$ in the sense of Definition~\ref{def:prefix-ultrametric} (including $d_{\mathrm{prod}}$).
Then $(X,\sigma,d)$ has the shadowing property if and only if it has the Lipschitz shadowing property.
Moreover, if shadowing holds, then Lipschitz shadowing holds with Lipschitz constant $L=1$.
\end{corollary}

\begin{proof}
Since $d$ is a prefix ultrametric, it induces the product topology and $d$ and $d_{\mathrm{prod}}$ are uniformly equivalent.

If $(X,\sigma,d)$ has Lipschitz shadowing, then it has shadowing by definition.
Conversely, assume $(X,\sigma,d)$ has shadowing. By uniform invariance of shadowing
(Proposition~\ref{prop:shadowing-uniform}), $(X,\sigma,d_{\mathrm{prod}})$ also has shadowing.
Therefore Theorem~\ref{thm:darji-finite-order} implies that $X$ is a shift of finite order $p$.
Now Lemma~\ref{lem:prefix-metric-satisfies-alpha-eta} shows that $d$ satisfies
\eqref{eq:alpha-eta-linear} with $C=1$, so Proposition~\ref{prop:finite-order-unif-Lip-sufficient}
yields Lipschitz shadowing for $(X,\sigma,d)$ with Lipschitz constant $L=1$.
\end{proof}

\begin{remark}
In the finite-alphabet case, $X$ is compact and $\sigma$ is positively expansive for $d_{\mathrm{prod}}$.
Sakai proved that for positively expansive maps, shadowing is equivalent to the existence of a compatible
metric with the Lipschitz shadowing property \cite[Theorem~1]{Sakai2003}. 
\end{remark}

\medskip\noindent
\textbf{Open question.}
Let $A$ be countably infinite and $X\subseteq A^{\N}$ with the product topology. Suppose $d$ is a metric
on $X$ that is \emph{uniformly equivalent} to $d_{\mathrm{prod}}$. If $(X,\sigma,d)$ has shadowing, must it
also have Lipschitz shadowing?  Equivalently, is Lipschitz shadowing stable under uniformly equivalent
changes of compatible metric in the non-compact product-topology setting?
We believe this is not true but were not able to find a counterexample.

\section{OTW full shift and the OTW metrics}

In this short section, we recall the construction of the OTW full shift and its metrics and show that these metrics are in fact ultrametrics. 
Throughout, we let $\A$ be a countably infinite discrete alphabet.

\subsection{The OTW full shift}\label{OTWfullsec}

\begin{definition}[Finite and infinite words; OTW full shift]
Let
\[
\fin:=\{\emptyword\}\cup\bigcup_{k\ge 1}\A^k,
\qquad
\A^{\N}=\{x_1x_2\cdots:x_i\in\A\}.
\]
The \emph{OTW full shift} is the disjoint union
\[
\full:=\A^{\N}\ \sqcup\ \fin.
\]
We write $|x|\in\{0,1,2,\dots,\infty\}$ for the length (with $|x|=\infty$ on $\A^{\N}$ and $|\emptyword|=0$).
\end{definition}

For $u\in\fin$ and $x\in\full$ we say $u\preceq x$ if $u$ is an initial segment of $x$.
OTW define generalized cylinder sets $Z(u,F)$ and show they form a basis of a compact, Hausdorff, totally disconnected topology on $\full$ \cite[\S2.2]{OTW2014}.

\subsection{The OTW metric}

\noindent\textbf{Prefix relation.}
For $u\in\fin$ and $x\in\full$, we write $u\preceq x$ (``$u$ is a prefix of $x$'') if either
$u=\emptyword$, or $u=u_1\cdots u_k$ with $k\ge1$ and $|x|\ge k$, and the first $k$ symbols of $x$
agree with $u$, i.e.\ $x_1\cdots x_k=u_1\cdots u_k$.

\begin{definition}[OTW metric attached to an enumeration]
\label{def:otwmetric}
Fix an enumeration $\mathcal{P}=\{p_1,p_2,\dots\}$ of $\fin$.
Define $d_{\mathcal P}:\full\times\full\to[0,1]$ by
\[
d_{\mathcal P}(x,y)=
\begin{cases}
0, & x=y,\\[0.3em]
2^{-i}, & i=\min\{j\ge 1:\ (p_j\preceq x)\ \oplus\ (p_j\preceq y)\},
\end{cases}
\]
where $\oplus$ denotes exclusive-or.
\end{definition}

\begin{proposition}\label{prop:otw-ultrametric}
For every enumeration $\mathcal P$, the function $d_{\mathcal P}$ is an \emph{ultrametric} on $\full$.
\end{proposition}

\begin{proof}
Fix an enumeration $\mathcal P=\{p_1,p_2,\dots\}$ of $\fin$.
For $x\in\full$, define its \emph{prefix indicator} function
\[
\chi_x:\fin\to\{0,1\},\qquad
\chi_x(u)=
\begin{cases}
1,& u\preceq x,\\
0,& u\npreceq x.
\end{cases}
\]
Thus $d_{\mathcal P}(x,y)=0$ iff $\chi_x=\chi_y$, and otherwise
\[
d_{\mathcal P}(x,y)=2^{-k(x,y)},
\quad\text{where}\quad
k(x,y):=\min\{j\ge1:\chi_x(p_j)\neq \chi_y(p_j)\}.
\]

Let $x,y,z\in\full$. If two of them are equal, the ultrametric inequality is trivial, so assume
$x,y,z$ are arbitrary and consider the indices
$k(x,y),k(x,z),k(z,y)\in\N\cup\{\infty\}$ (with $k=\infty$ meaning equality).
Set
\[
r:=\min\{k(x,z),k(z,y)\}.
\]
By definition of $r$, we have
\[
\chi_x(p_j)=\chi_z(p_j)\quad\text{and}\quad \chi_z(p_j)=\chi_y(p_j)
\qquad\text{for all } j<r.
\]
Hence $\chi_x(p_j)=\chi_y(p_j)$ for all $j<r$, which implies
\[
k(x,y)\ \ge\ r.
\]
Therefore,
\[
d_{\mathcal P}(x,y)=2^{-k(x,y)}\ \le\ 2^{-r}
=\max\{2^{-k(x,z)},2^{-k(z,y)}\}
=\max\{d_{\mathcal P}(x,z),d_{\mathcal P}(z,y)\},
\]
which is precisely the ultrametric inequality. 
\end{proof}

\begin{remark}\label{prop:otw-topology}
For every enumeration $\mathcal P$, the metric $d_{\mathcal P}$ induces the OTW topology on $\full$
(see \cite[\S2.2--\S2.3]{OTW2014}).
\end{remark}

\subsection{OTW shift and Deaconu--Renault viewpoint}

The shift map $\sigma:\full\to\full$ was defined in \cite[Definition~2.22]{OTW2014} as:
\[
\sigma(x_1x_2x_3\cdots)=x_2x_3\cdots,\quad
\sigma(x_1\cdots x_k)=x_2\cdots x_k\ (k\ge 2),\quad
\sigma(x)=\emptyword\ (|x|\le 1).
\]
OTW prove that $\sigma$ is continuous at every point except $\emptyword$, and $\sigma$ restricts to a local homeomorphism on $\full\setminus\{\emptyword\}$ \cite[Proposition~2.23]{OTW2014}. Thus, up to the standard domain restriction, i.e. restriction to elements of length greater than or equal to 1, $(\full,\sigma)$ fits the Deaconu--Renault framework \cite[Definition~2.1]{GoncalvesUggioni2025Shadowing}. This framework of a Deaconu--Renault system is more friendly to C*-algebraic applications and it is used, for example, in \cite{BoavaCastroGoncalvesvanWyk2024, BoavaCastroGoncalvesvanWyk2025}.

\section{Shadowing for the OTW full shift}\label{OTWshadowing}

\subsection{All OTW metrics are uniformly equivalent}

\begin{lemma}
\label{lem:uniform}
If $\mathcal P$ and $\mathcal Q$ are two enumerations of $\fin$, then $d_{\mathcal P}$ and $d_{\mathcal Q}$ are uniformly equivalent.
\end{lemma}

\begin{proof}
By OTW, both metrics induce the same compact Hausdorff topology on $\full$.
Hence $\mathrm{id}:(\full,d_{\mathcal P})\to(\full,d_{\mathcal Q})$ is continuous.
Since $(\full,d_{\mathcal P})$ is compact, Heine--Cantor implies $\mathrm{id}$ is uniformly continuous.
Symmetrically, $\mathrm{id}:(\full,d_{\mathcal Q})\to(\full,d_{\mathcal P})$ is uniformly continuous.
\end{proof}

\subsection{Shadowing for the OTW full shift}

\begin{proposition}
\label{prop:shadowingfull}
Let $\A$ be countably infinite. For \emph{every} OTW metric $d_{\mathcal P}$, the system $(\full,\sigma,d_{\mathcal P})$ has the shadowing property.
\end{proposition}

\begin{proof}
\medskip

Throughout this proof we interpret pseudo-orbits and shadowing in the Deaconu--Renault sense for the
local homeomorphism $\sigma\colon \full\setminus\{\emptyword\}\to\full$, so all iterates
$\sigma^n(x)$ that appear are understood to be defined. With this convention, the shadowing results of \cite{GoncalvesUggioni2025Shadowing} apply directly.

In \cite[Proposition~5.1]{GoncalvesUggioni2025Shadowing}, the authors prove shadowing for the Deaconu--Renault system associated to the \emph{rose with infinitely many petals} (one vertex, countably many loops), endowed with the natural ultrametric built from an edge-ordered enumeration (see \cite[Example~3.9]{GoncalvesUggioni2025Shadowing}).
The associated boundary path space identifies canonically with the OTW full shift $\full$ (finite paths correspond to finite words, and infinite paths correspond to infinite sequences), and the metric in \cite{GoncalvesUggioni2025Shadowing} agrees with an OTW metric $d_{\mathcal P_0}$ for a particular enumeration $\mathcal P_0$.

Hence $(\full,\sigma,d_{\mathcal P_0})$ has shadowing.
Now let $\mathcal P$ be any enumeration. By Lemma~\ref{lem:uniform}, $d_{\mathcal P}$ and $d_{\mathcal P_0}$ are uniformly equivalent.
By Proposition~\ref{prop:shadowing-uniform}, shadowing transfers from $d_{\mathcal P_0}$ to $d_{\mathcal P}$.
\end{proof}

\section{Lipschitz shadowing for OTW-shift spaces}\label{OTW lip shad}

We now construct an enumeration $\mathcal Q$ such that the induced OTW metric yields failure of Lipschitz shadowing, even though shadowing holds by Proposition~\ref{prop:shadowingfull}. We also construct an enumeration where the Lipschitz shadowing property holds.

\subsection{Non-Lipschitz shadowing enumeration}

We begin with the definition of the properties our desired enumeration should satisfy.

\begin{lemma}\label{lem:bad-enum}
There exist an increasing sequence $i_1<i_2<\cdots$ in $\N$ and an enumeration
$\mathcal Q=\{q_k\}_{k\ge 1}$ of $\fin$ with the following property:
for each $n\ge 1$ there exist \emph{pairwise distinct} letters
$a^{(n)}_1,a^{(n)}_2,a^{(n)}_3,b^{(n)}\in\A$ such that
\begin{enumerate}[label=(\roman*),leftmargin=2.2em]
\item $q_{i_n-n}=a^{(n)}_1a^{(n)}_2a^{(n)}_3$;
\item no word among $q_1,\dots,q_{i_n}$ starts with $a^{(n)}_2$;
\item the letter $b^{(n)}$ does not appear in \emph{any} word among $q_1,\dots,q_{i_n}$.
\end{enumerate}
\end{lemma}

\begin{proof}
Fix once and for all a listing
\[
\fin=\{w_1,w_2,w_3,\dots\}.
\]
We construct inductively an increasing sequence $(i_n)_{n\ge1}$ and the values of an enumeration
$\mathcal Q=\{q_k\}_{k\ge1}$.

\medskip\noindent
\textbf{Step $n=1$ (initial step).}
Choose pairwise distinct letters
\[
a^{(1)}_1,\ a^{(1)}_2,\ a^{(1)}_3,\ b^{(1)}\in\A .
\]
Choose an integer $i_1\ge2$ (so that $i_1-1>0=i_0$) and define
\[
q_{\,i_1-1}:=a^{(1)}_1a^{(1)}_2a^{(1)}_3.
\]
For every other index $k\in (0,i_1]\setminus\{i_1-1\}$ (i.e.\ $1\le k\le i_1$ and $k\neq i_1-1$),
define $q_k$ by taking, among the words $w_j$ not used yet, the first ones (in increasing order of $j$)
satisfying:
\begin{itemize}
\item $w_j$ does not start with $a^{(1)}_2$, and
\item $b^{(1)}$ does not occur in $w_j$,
\end{itemize}
until all such indices are filled.

This is possible because there are infinitely many words in $\fin$ satisfying these two restrictions, and we need to fill only finitely
many slots.
By construction, items (i)--(iii) of the lemma hold for $n=1$.

\medskip\noindent
\textbf{Induction hypothesis.}
Assume that for some $n\ge2$ we have already chosen indices
\[
i_1<i_2<\cdots<i_{n-1},
\]
and defined words $q_1,\dots,q_{i_{n-1}}$ (all distinct), such that for each $r=1,\dots,n-1$ there exist
pairwise distinct letters $a^{(r)}_1,a^{(r)}_2,a^{(r)}_3,b^{(r)}\in\A$ for which conditions (i)--(iii)
hold at stage $r$ (i.e.\ up to the cutoff $i_r$).

\medskip\noindent
\textbf{Inductive step (construct stage $n$).}
Let
\[
S_{n-1}:=\{\text{letters that occur in some }q_k,\ 1\le k\le i_{n-1}\}.
\]
Since only finitely many words have been defined so far, $S_{n-1}$ is finite. Choose pairwise distinct letters
\[
a^{(n)}_1,\ a^{(n)}_2,\ a^{(n)}_3,\ b^{(n)}\in \A\setminus S_{n-1}.
\]
Now choose an integer $i_n$ so large that
\[
i_n-n>i_{n-1}.
\]
Define the special slot
\[
q_{\,i_n-n}:=a^{(n)}_1a^{(n)}_2a^{(n)}_3.
\]
This word is new, since the letters $a^{(n)}_1,a^{(n)}_2,a^{(n)}_3$ were chosen outside $S_{n-1}$.

For every other index
\[
k\in (i_{n-1},i_n]\setminus\{i_n-n\},
\]
define $q_k$ by taking, among the words $w_j$ not used yet, the first ones (in increasing order of $j$)
satisfying:
\begin{itemize}
\item $w_j$ does not start with $a^{(n)}_2$, and
\item $b^{(n)}$ does not occur in $w_j$,
\end{itemize}
until all such indices are filled.

This is possible because there are infinitely many words in $\fin$ satisfying these two restrictions
(e.g.\ one-letter words $(c)$ with $c\in\A\setminus\{a^{(n)}_2,b^{(n)}\}$), and only finitely many
words have been used so far, so infinitely many eligible words remain unused.

With this construction, items (ii) and (iii) hold at stage $n$:
\begin{itemize}
\item[(ii)] no word among $q_1,\dots,q_{i_n}$ starts with $a^{(n)}_2$, since $a^{(n)}_2\notin S_{n-1}$ and all newly assigned words were chosen not to start with $a^{(n)}_2$;
\item[(iii)] the letter $b^{(n)}$ does not appear in any word among $q_1,\dots,q_{i_n}$, since $b^{(n)}\notin S_{n-1}$ and all newly assigned words were chosen to avoid $b^{(n)}$.
\end{itemize}
This completes the inductive construction.

\medskip\noindent
\textbf{Why $\mathcal Q$ enumerates $\fin$.}
We first note that by construction all $q_k$ are distinct, and every $q_k$ is some $w_j$, so $\{q_k\}\subseteq\fin$.
It remains to show that every $w_j$ appears.

For each $n$, set $B_n:=\{a^{(n)}_2,b^{(n)}\}$. Since at each stage the letters are chosen fresh,
the sets $B_n$ are pairwise disjoint. Let $w\in\fin$, and let $\mathcal L(w)\subseteq \A$ be the finite set of letters
occurring in $w$. Because the $B_n$ are disjoint and $\mathcal L(w)$ is finite, there are only finitely many $n$ with
$\mathcal L(w)\cap B_n\neq\varnothing$. Hence there exists $N(w)$ such that for every $n\ge N(w)$:
\begin{itemize}
\item $w$ does not start with $a^{(n)}_2$, and
\item $b^{(n)}$ does not occur in $w$.
\end{itemize}
Thus $w$ is \emph{eligible} to be used as a filler word at every stage $n\ge N(w)$.

Now fix $j\ge1$ and consider $w_j$. Let
\[
N_j:=N(w_j).
\]
For every stage $n\ge N_j$, the word $w_j$ is eligible. Moreover, our choice $i_n-n>i_{n-1}$ implies
$i_n-i_{n-1}>n$, hence the number of filler slots at stage $n$ is
\[
\#\Big((i_{n-1},i_n]\setminus\{i_n-n\}\Big)=i_n-i_{n-1}-1\ \ge\ n.
\]
Therefore the total number of filler slots available from stage $N_j$ onward is infinite. Since at each stage we fill
filler slots with the \emph{earliest} eligible unused words in the fixed list $\{w_1,w_2,\dots\}$, it follows that once
all eligible words $w_1,\dots,w_{j-1}$ have been used, the word $w_j$ becomes the earliest eligible unused word and must be
selected at some later filler slot. Hence $w_j$ is eventually chosen.

Since $j$ was arbitrary, every $w_j$ appears exactly once among the $q_k$. Thus $\mathcal Q$ is an enumeration of $\fin$,
and the lemma follows.
\end{proof}

\subsection{Failure of Lipschitz shadowing}

\begin{theorem}\label{thm:lipschitz-fails}
Let $\mathcal Q$ be as in Lemma~\ref{lem:bad-enum}, and let $d_{\mathcal Q}$ be the associated OTW metric.
Then $(\full,\sigma,d_{\mathcal Q})$ does \emph{not} have the Lipschitz shadowing property.
\end{theorem}

\begin{proof}
We prove the negation of Lipschitz shadowing: for every $L>0$ and every $\delta_0>0$,
there exist $\delta\in(0,\delta_0]$ and a $\delta$-pseudo-orbit in $\full$ that is not
$L\delta$-shadowed by any point of $\full$.

Fix $L>0$ and $\delta_0>0$.
Choose $n\ge1$ such that $2^n>L$.
By Lemma~\ref{lem:bad-enum}, choose the corresponding index $i_n$ and letters
\[
a_1:=a^{(n)}_1,\qquad a_2:=a^{(n)}_2,\qquad a_3:=a^{(n)}_3,\qquad b:=b^{(n)}.
\]
By increasing $n$ further if necessary, we may assume
\[
2^{-i_n}<\delta_0.
\]
Set
\[
\delta:=2^{-i_n},
\qquad
w:=a_1a_2a_3=q_{i_n-n}.
\]

\medskip\noindent
\textbf{Step 1: construct a $\delta$-pseudo-orbit.}
Define $\{x^m\}_{m\ge0}\subseteq \A^{\N}\subseteq\full$ by
\[
x^{2k}:=a_1a_2a_3\, bbb\cdots,
\qquad
x^{2k+1}:=b\,a_1a_2a_3\, bbb\cdots
\quad (k\ge0).
\]

We claim that $\{x^m\}$ is a $\delta$-pseudo-orbit.

For odd $m=2k+1$,
\[
\sigma(x^{2k+1})=a_1a_2a_3\, bbb\cdots = x^{2k+2},
\]
so
\[
d_{\mathcal Q}(\sigma(x^{2k+1}),x^{2k+2})=0<\delta.
\]

For even $m=2k$,
\[
\sigma(x^{2k})=a_2a_3\, bbb\cdots,
\qquad
x^{2k+1}=b\,a_1a_2a_3\, bbb\cdots.
\]
Any nonempty prefix of $\sigma(x^{2k})$ starts with $a_2$, while any nonempty prefix of $x^{2k+1}$
starts with $b$.

By Lemma~\ref{lem:bad-enum}(ii), no word among $q_1,\dots,q_{i_n}$ starts with $a_2$.
By Lemma~\ref{lem:bad-enum}(iii), the letter $b$ does not occur in any word among $q_1,\dots,q_{i_n}$,
hence in particular no such word starts with $b$.
Therefore, for every $r\le i_n$, the word $q_r$ is a prefix of $\sigma(x^{2k})$ if and only if it is a
prefix of $x^{2k+1}$ (indeed, both are false unless $q_r=\emptyword$, in which case both are true).
So the first distinguishing index is $>i_n$, and hence
\[
d_{\mathcal Q}(\sigma(x^{2k}),x^{2k+1})<2^{-i_n}=\delta.
\]

Thus $\{x^m\}_{m\ge0}$ is a $\delta$-pseudo-orbit.

\medskip\noindent
\textbf{Step 2: no point can $L\delta$-shadow it.}
Assume, toward a contradiction, that there exists $x\in\full$ such that
\[
d_{\mathcal Q}(\sigma^m(x),x^m)\le L\delta
\qquad\text{for all }m\ge0.
\]
Since $2^n>L$, we have
\[
L\delta < 2^n\delta = 2^{-(i_n-n)}.
\]
Hence, for every $m\ge0$,
\[
d_{\mathcal Q}(\sigma^m(x),x^m) < 2^{-(i_n-n)}.
\]
By definition of the OTW metric and since $q_{i_n-n}=w$, this implies
\[
w\preceq \sigma^m(x) \iff w\preceq x^m
\qquad\text{for every }m\ge0.
\]

In particular, for every $k\ge0$, the point $x^{2k}$ begins with $w$, so
\[
w\preceq \sigma^{2k}(x)\qquad\text{for all }k\ge0.
\]
Taking $k=0$ and $k=1$, we get
\[
w\preceq x
\qquad\text{and}\qquad
w\preceq \sigma^2(x).
\]
Writing $w=a_1a_2a_3$, the first condition says that the first three letters of $x$ are
\[
x_1x_2x_3=a_1a_2a_3,
\]
while the second says that the first three letters of $\sigma^2(x)$ are also $a_1a_2a_3$, i.e.
\[
x_3x_4x_5=a_1a_2a_3.
\]
Comparing the first letters of these two blocks gives
\[
x_3=a_3 \quad\text{and}\quad x_3=a_1,
\]
hence $a_1=a_3$, contradicting the fact that $a_1$ and $a_3$ are distinct.

Therefore no point of $\full$ can $L\delta$-shadow the pseudo-orbit $\{x^m\}$.
Since $L>0$ and $\delta_0>0$ were arbitrary, $(\full,\sigma,d_{\mathcal Q})$ does not have the
Lipschitz shadowing property.
\end{proof}
\subsection{An enumeration inducing Lipschitz shadowing}
Let $\mathcal P=\{p_1,p_2,\dots\}$ be an enumeration of $\fin$.
Given $i\in\N$, we write
\[
a \not\preceq \{p_1,\dots,p_i\}
\quad\text{if}\quad
a \text{ does not occur in any of the words } p_1,\dots,p_i .
\]

\begin{definition}[Prefix--shift compatible enumeration]\label{def:pscomp}
An enumeration $\mathcal P=\{p_1,p_2,\dots\}$ of $\fin$ is \emph{prefix--shift compatible} if
$p_1=\emptyword$ and the following hold:
\begin{enumerate}[label=(\arabic*),leftmargin=2.2em]
\item (\emph{prefix-closed}) If $p_k=uv$ with $u\in\fin$ and $v\in\fin\setminus\{\emptyword\}$,
then $u=p_m$ for some $m<k$.
\item (\emph{shift-closed}) If $|p_k|\ge 2$, then $\sigma(p_k)=p_m$ for some $m<k$.
\end{enumerate}
\end{definition}

\begin{remark}\label{rem:exist-pscomp}
A prefix--shift-compatible enumeration exists. Indeed, fix an enumeration
$\A=\{a_1,a_2,\dots\}$ and, for each $N\ge1$, let
\[
B_N:=\{u\in\fin:\ |u|\le N \text{ and every letter of }u\text{ belongs to }\{a_1,\dots,a_N\}\}.
\]
Then each $B_N$ is finite, $B_N\subseteq B_{N+1}$, and $\bigcup_{N\ge1}B_N=\fin$.

Enumerate $\fin$ block by block, listing first $B_1$, then the new words in $B_2\setminus B_1$,
then those in $B_3\setminus B_2$, and so on. Inside each finite block $B_N\setminus B_{N-1}$,
choose any order extending the finite dependency relation generated by:
\begin{itemize}
\item proper prefixes (list every proper prefix of $u$ before $u$), and
\item the left shift (if $|u|\ge2$, list $\sigma(u)$ before $u$ whenever $\sigma(u)$ lies in the same block).
\end{itemize}
This is possible because the dependency relation is acyclic (it strictly decreases word length).

Now let $u$ be any word appearing in some block $B_N\setminus B_{N-1}$. Every proper prefix of $u$
and, if $|u|\ge2$, the word $\sigma(u)$ belong to $B_N$ (they use the same letters and have smaller length),
hence they are listed either in an earlier block or earlier in the same block. Therefore the resulting
enumeration is prefix--shift-compatible. In particular, it is shift-compatible.
\end{remark}

\begin{proposition}[Lipschitz shadowing for a prefix--shift-compatible OTW metric]\label{prop:lipschitz-holds}
Let $\A$ be countably infinite, and let $\full$ be the OTW full shift.
Fix a prefix--shift-compatible enumeration $\mathcal P=\{p_1,p_2,\dots\}$ of $\fin$
(see Remark~\ref{rem:exist-pscomp}) and write $d_0:=d_{\mathcal P}$.
Then $(\full,\sigma,d_0)$ has the Lipschitz shadowing property.
In fact, one may take Lipschitz constant $L=2$.
\end{proposition}

\begin{proof}
Let $\delta_0:=1/4$ and fix $0<\delta\le \delta_0$.
Choose $j\ge 2$ such that
\begin{equation}\label{eq:choosej-L2}
2^{-j}\le \delta < 2^{-(j-1)}.
\end{equation}
Set
\[
N:=j-1,
\qquad
F:=\{p_1,\dots,p_N\}.
\]
Let $\{x^n\}_{n\ge0}$ be a $\delta$-pseudo-orbit, i.e.
\[
d_0(\sigma(x^n),x^{n+1})<\delta
\qquad (n\ge0).
\]
Since $d_0$ takes values in $\{0\}\cup\{2^{-k}:k\in\N\}$, \eqref{eq:choosej-L2} implies
\[
d_0(\sigma(x^n),x^{n+1})<2^{-N}
\qquad (n\ge0).
\]
Hence, by the definition of $d_0$, we have the prefix agreement
\begin{equation}\label{eq:prefix-agree-L2}
\forall\,u\in F,\ \forall\,n\ge0,\qquad
u\preceq \sigma(x^n)\iff u\preceq x^{n+1}.
\end{equation}

For each $n\ge0$, define
\[
S_n:=\{u\in F:\ u\preceq x^n\}.
\]
Because the enumeration is prefix--shift-compatible, the finite set $F$ has the following closure properties:
\begin{itemize}
\item if $u\in F$ and $v$ is a proper prefix of $u$, then $v\in F$;
\item if $u\in F$ and $|u|\ge 2$, then $\sigma(u)\in F$.
\end{itemize}

\medskip\noindent
\textbf{Step 1: define a candidate shadowing point from one-letter information.}
Choose a letter $z\in\A$ such that
\[
z \not\preceq \{p_1,\dots,p_N\},
\]
i.e.\ $z$ does not occur in any of the words $p_1,\dots,p_N$.
(This is possible because only finitely many letters occur in $p_1,\dots,p_N$.)

Define $x=x_1x_2x_3\cdots\in\A^{\N}\subseteq\full$ recursively by:
for each $n\ge0$,
\[
x_{n+1}:=
\begin{cases}
a, & \text{if there exists a (necessarily unique) one-letter word } (a)\in F \text{ with } (a)\in S_n,\\
z, & \text{if no one-letter word in }F\text{ belongs to }S_n.
\end{cases}
\]
Uniqueness is clear because a point of $\full$ has at most one one-letter prefix.

\medskip\noindent
\textbf{Step 2: matching prefix tests on the finite set $F$.}
We claim that for every $n\ge0$ and every $u\in F$,
\begin{equation}\label{eq:main-claim-L2}
u\preceq \sigma^n(x)\iff u\in S_n
\qquad\bigl(\text{equivalently, }u\preceq \sigma^n(x)\iff u\preceq x^n\bigr).
\end{equation}
We prove this by induction on $|u|$.

\smallskip\noindent
\emph{Base case $|u|=0$.}
Then $u=\emptyword$, and $\emptyword\preceq y$ for every $y\in\full$, so \eqref{eq:main-claim-L2} is trivial.

\smallskip\noindent
\emph{Base case $|u|=1$.}
Let $u=(a)\in F$. Then
\[
u\preceq \sigma^n(x)\iff x_{n+1}=a.
\]
By the definition of $x_{n+1}$, this is equivalent to $(a)\in S_n$.
Hence \eqref{eq:main-claim-L2} holds for all $u\in F$ with $|u|=1$.

\smallskip\noindent
\emph{Inductive step.}
Assume \eqref{eq:main-claim-L2} holds for all words in $F$ of length at most $\ell-1$, where $\ell\ge2$,
and let $u\in F$ with $|u|=\ell$.
Write
\[
u=c\,v,
\qquad c\in\A,\quad v=\sigma(u)\in\fin.
\]
Since the enumeration is prefix--shift-compatible, we have $(c)\in F$ (prefix closure) and $v\in F$ (shift closure).

Now, for every $n\ge0$,
\begin{align*}
u\preceq \sigma^n(x)
&\iff \bigl((c)\preceq \sigma^n(x)\bigr)\ \text{and}\ \bigl(v\preceq \sigma^{n+1}(x)\bigr)\\
&\iff \bigl((c)\in S_n\bigr)\ \text{and}\ \bigl(v\in S_{n+1}\bigr),
\end{align*}
where the second equivalence uses the base case (for $(c)$) and the induction hypothesis (for $v$).

On the other hand, since $u=c\,v$,
\[
u\preceq x^n
\iff \bigl((c)\preceq x^n\bigr)\ \text{and}\ \bigl(v\preceq \sigma(x^n)\bigr).
\]
Because $(c)\in F$, the first condition is equivalent to $(c)\in S_n$.
Because $v\in F$, \eqref{eq:prefix-agree-L2} gives
\[
v\preceq \sigma(x^n)\iff v\preceq x^{n+1}
\iff v\in S_{n+1}.
\]
Therefore
\[
u\preceq x^n
\iff \bigl((c)\in S_n\bigr)\ \text{and}\ \bigl(v\in S_{n+1}\bigr),
\]
which matches the criterion obtained for $u\preceq \sigma^n(x)$.
Hence \eqref{eq:main-claim-L2} holds for $u$, completing the induction.

\medskip\noindent
\textbf{Step 3: estimate the shadowing error.}
Fix $n\ge0$. By \eqref{eq:main-claim-L2}, for every $i\le N$ we have
\[
p_i\preceq \sigma^n(x)\iff p_i\preceq x^n.
\]
Thus the first index at which the prefix tests for $\sigma^n(x)$ and $x^n$ differ is \(>N\), and therefore
\[
d_0(\sigma^n(x),x^n)<2^{-N}=2^{-(j-1)}\le 2\delta.
\]
Since this holds for every $n\ge0$, the point $x$ $2\delta$-shadows the given $\delta$-pseudo-orbit.

Therefore $(\full,\sigma,d_0)$ has the Lipschitz shadowing property with Lipschitz constant $L=2$.
\end{proof}



\bibliographystyle{abbrv}
\bibliography{ref1}

\end{document}